\documentclass[final,onefignum,onetabnum]{siamart171218}

\usepackage{amsmath,amssymb,amsfonts}
\usepackage{algorithm,algorithmic}
\usepackage{graphicx}
\usepackage{textcomp}

\usepackage{color}
\usepackage{soul}

\usepackage{lipsum}
\usepackage{epstopdf}

\usepackage{amsopn}
\DeclareMathOperator{\diag}{diag}

\usepackage{enumerate}
\usepackage{cancel}
\usepackage{mathrsfs}
\usepackage{mathdots}
\usepackage{euscript}
\usepackage{amscd}
\usepackage{cite}
\usepackage{placeins}
\usepackage{tikz}
\usetikzlibrary{snakes,arrows,shapes}

\graphicspath{{images/}}

\newcommand{\bmat}{\left[ \begin{matrix}}
\newcommand{\emat}{\end{matrix} \right]}
\newcommand{\innerprod}[2]{\langle{#1},\,{#2}\rangle}

\DeclareMathOperator{\trace}{tr}
\DeclareMathOperator{\adjugate}{adj}

\DeclareMathOperator{\E}{{\mathbb E}}
\newcommand{\Rbb}{\mathbb R}

\newcommand{\Cbb}{\mathbb C}

\newcommand{\Zbb}{\mathbb Z}

\newcommand{\Tbb}{\mathbb T}

\newcommand{\yb}{\mathbf  y}
\newcommand{\zb}{\mathbf  z}

\newcommand{\vb}{\mathbf  v}

\newcommand{\ab}{\mathbf a}

\newcommand{\ub}{\mathbf  u}

\newcommand{\qb}{\mathbf q}
\newcommand{\tb}{\mathbf t}
\newcommand{\kb}{\mathbf k}

\newcommand{\zerob}{\mathbf 0}

\newcommand{\Qb}{\mathbf Q}

\newcommand{\Yb}{\mathbf Y}

\newcommand{\Xb}{\mathbf  X}

\newcommand{\thetab}{\boldsymbol{\theta}}

\newcommand{\Sigmab}{\boldsymbol{\Sigma}}

\DeclareMathOperator{\interior}{int}
\newcommand{\Hfrak}{\mathfrak{H}}

\newcommand{\Cfrak}{\mathfrak{C}}

\newcommand{\Sfrak}{\mathfrak{S}}

\newcommand{\Lscr}{\mathscr{L}}

\newcommand{\Lcal}{\mathcal{L}}
\newcommand{\Zcal}{\mathcal{Z}}
\newcommand{\Dcal}{\mathcal{D}}

\renewcommand{\d}{\mathrm{d}}

\renewcommand{\a}{\mathrm{a}}
\newcommand{\s}{\mathrm{s}}

\newcommand{\m}{\mu} % symbol for the Lebesgue measure

\usepackage{lipsum}
\usepackage{amsfonts}
\usepackage{graphicx}
\usepackage{epstopdf}
\usepackage{algorithmic}

\ifpdf
  \DeclareGraphicsExtensions{.eps,.pdf,.png,.jpg}
\else
  \DeclareGraphicsExtensions{.eps}
\fi

\newsiamremark{remark}{Remark}
\newsiamremark{hypothesis}{Hypothesis}
\crefname{hypothesis}{Hypothesis}{Hypotheses}
\newsiamthm{claim}{Claim}
\newsiamthm{assumption}{Assumption}
\newsiamthm{conjecture}{Conjecture}

\newsiamremark{example}{Example}
\newsiamremark{problem}{Problem}
\headers{M$^2$-Spectral Estimation Using the Tau-Divergence}{B. Zhu, A. Ferrante, J. Karlsson, and M. Zorzi}

\title{M$^2$-Spectral Estimation: A Flexible Approach Ensuring Rational Solutions\thanks{Submitted to the editors April 30, 2020.
\funding{This work was supported by the SID project ``A Multidimensional and Multivariate Moment Problem Theory for Target Parameter Estimation in Automotive Radars'' (ZORZ\_SID19\_01) funded by the Department of Information Engineering of the University of Padova. The first author was also partially supported by the ``Hundred-Talent Program'' of the Sun Yat-sen University.}}}

\author{Bin Zhu\thanks{School of Intelligent Systems Engineering, Sun Yat-sen University, Waihuan East Road 132, 510006 Guangzhou, China (\email{zhub26@mail.sysu.edu.cn}).}
\and Augusto Ferrante\thanks{Department of Information Engineering, University of Padova, Via Gradenigo 6/B, 35131 Padova, Italy (\email{augusto@dei.unipd.it}, \email{zorzimat@dei.unipd.it}).}
\and Johan Karlsson\thanks{Division of Optimization and Systems Theory, Department of Mathematics, KTH Royal Institute of Technology, 10044 Stockholm, Sweden (\email{johan.karlsson@math.kth.se}).}
\and Mattia Zorzi\footnotemark[3]
}

\usepackage{amsopn}
\begin{document}

\maketitle

\begin{abstract}
This paper concerns a spectral estimation problem for multivariate (i.e., vector-valued) signals defined on a multidimensional domain, abbreviated as M$^2$.
The problem is posed as solving a finite number of trigonometric moment equations for a nonnegative matricial measure, which is well known as the \emph{covariance extension problem} in the literature of systems and control.
This inverse problem and its various generalizations have been extensively studied in the past three decades, and they find applications in diverse fields such as modeling and system identification, signal and image processing, robust control, circuit theory, etc.
In this paper, we address the challenging M$^2$ version of the problem, and elaborate on a solution technique via convex optimization with the $\tau$-divergence family.
{As a major contribution of this work}, we show that by properly choosing the parameter of the divergence index, the optimal spectrum is {a rational function}, that is, the solution is a spectral density which {can be represented by a finite-dimensional system, as} desired in {many} practical applications.
\end{abstract}


% REQUIRED

\section{Introduction}\label{sec:intro}
In this paper, we address the problem of estimating a multidimensional and multivariate (M$^2$) spectrum which characterizes a 
{second-order stationary} random field. {Such models are particularly useful when considering high-dimensional stochastic processes that are stationary with respect to some of the dimensions which are then taken as the domain. Applications of this appear in, e.g., system identification, image processing, and radar signal processing
\cite{bose2003multidimensional,ZFKZ2019fusion}}. 
{Here we deal with} the spectral estimation problem using a moment-based approach. Assume that we have computed from the data a finite number of covariances of the random field, and a prior is available, i.e., a spectrum embedding the a priori information that we have. In particular, if we have no prior knowledge, we can take the spectrum of a white noise. Then, the M$^2$ spectral estimator is the closest spectrum to the prior satisfying the moment conditions. The closeness between the solution and the prior is measured by a divergence index (or pseudo-distance). In \cite{ZFKZ2019M2}, we showed that such a problem is well-posed for a periodic field using the Itakura-Saito distance \cite{enqvist2008minimal,FMP-12}. {For a comprehensive treatment of multidimensional moment problems, we refer the reader to some recent books \cite{lasserre2010moments,bakonyi2011matrix,schmudgen2017moment}.}

The moment-based approach for spectral estimation with prior has been widely studied in the unidimensional and scalar case \cite{BGL-THREE-00,enqvist2004aconvex,BGL-98,LPcirculant-13,karlsson2010theinverse,zhu2018wellposed}, as well as its multivariate extension \cite{FMP-12,FPR-08,RFP-09,lindquist2013onthemultivariate,Zhu-Baggio-19,Zorzi-F-12,FPZ-12}. These optimization problems differ by the considered divergence {indices} and do admit a unique spectral density as solution. It is worth noting that different divergence indices lead to solutions with different properties, e.g., their {complexities} in terms of the McMillan degree. Interestingly, these divergences are connected through the $\alpha$-divergence \cite{Z-14rat}, the $\beta$-divergence \cite{Z-14} and the $\tau$-divergence \cite{Z-15,zorzi2015interpretation}. The moment-based approach equipped with a divergence family is very flexible in the sense that we obtain a family of solutions, each corresponding to a particular value of the parameter of the family, and we can choose one of them depending on the features that we would like to have.

The multidimensional extension of the moment-based spectral estimation approach, however, has been less studied. We mention \cite{Georgiou-06} and a recent work \cite{ZFKZ2019M2} in which discrete spectra are considered.
While in the latter case it is possible to show that there exists a unique spectral density which solves the optimization problem, the problem becomes more challenging when the spectrum is supported on the whole multidimensional frequency domain. One of the difficulties due to the multiple dimensionality can be seen from \cite{ringh2015multidimensional,KLR-16multidimensional,RKL-16multidimensional,ringh2018multidimensional} where
the scalar multidimensional problem has been investigated. It is shown that in general, the solution to the constrained optimization problem is not necessarily a spectral density, but rather a \emph{spectral measure} that may contain a singular part. The latter is not desirable in most applications. Indeed, as it has been shown in \cite{KLR-16multidimensional}, it is difficult to characterize the singular measure and it is in general not unique.

The aim of this paper is to propose a M$^2$ spectral estimator based on the $\tau$-divergence family.
We show {via duality analysis} that the corresponding {dual optimization} problem admits a unique solution.
Furthermore, the flexibility using the $\tau$-divergence family guarantees that the family of solutions to the primal problem contains at least one \emph{rational} M$^2$ spectral density {(see \cref{sec:integrability})}.
The latter is the most important result {in this work} and it has never been addressed in the multidimensional case.
The significance of rationality is well understood in the unidimensional case since one can construct via \emph{spectral factorization} a digital filter which produces a process with prescribed second-order statistics when fed with white noise. Although spectral factorization is not always possible in the multidimensional setting, rationality still seems to be a key ingredient towards a finite-dimensional realization theory.

The outline of the paper is as follows. \Cref{sec:background} gives some background material on moment-based approach for spectral estimation. In \cref{sec:problem} we formulate the M$^2$ spectral estimation problem using the $\tau$-divergence. In \cref{sec:duality} we derive the corresponding dual problem. In \cref{sec:solution_dual} we prove the existence and uniqueness of the solution to the dual problem. In \cref{sec:integrability} we show that, for a suitable choice of $\tau$, the corresponding solution to the primal problem is unique and it is a spectral density. If the aforementioned condition is not satisfied, then we {show in \cref{sec:measure} that} the primal solution is a spectral measure that may contain a singular part. Finally, in \cref{sec:conclusions} we draw the conclusions.

\section{Background}\label{sec:background}

Consider the spectral estimation problem for a zero-mean second-order stationary random field $\yb(\tb)$ whose index $\tb=(t_1,t_2,\dots,t_d)$ lives in $\Zbb^d$. Here the dimension $d$ is a positive integer. 
At any fixed $\tb$, $\yb(\tb)$ is a zero-mean complex random vector of dimension $m$. Stationarity means that the covariance $\Sigma_{\kb}:=\E\,\yb(\tb+\kb)\yb(\tb)^*$ does not depend on $\tb$. {Such a definition} implies the symmetry $\Sigma_{-\kb}=\Sigma_{\kb}^*$. 
The spectral density of the random field is defined as the multidimensional Fourier transform of the covariance field
\begin{equation}\label{Phi_spec_density}
\Phi(e^{i\thetab}):=\sum_{\kb\in\Zbb^d} \Sigma_{\kb}e^{-i\innerprod{\kb}{\thetab}},
\end{equation}
where {$\innerprod{\kb}{\thetab}:= \sum_{j=1}^{d} k_j\theta_j$ is the standard inner product and} the (angular) frequency vector, $\thetab=(\theta_1,\theta_2,\dots,\theta_d)\in\Tbb^d:=(-\pi,\pi]^d$, {can be identified with the corresponding point $e^{i\thetab}:=(e^{i\theta_1},\dots,e^{i\theta_d})$ on the $d$-torus.}\footnote{{The symbol $\Tbb$ is used in a {somewhat} nonstandard way to denote the interval $(-\pi,\pi]$ for the convenience of writing multidimensional integrals. Also, here we treat the spectral density as a function of $\thetab$ but keep the conventional notation $\Phi(e^{i\thetab})$.}} 
If the above Fourier transform exists, then by the Wiener--Khinchin theorem, the spectral density $\Phi(e^{i\thetab})$ is positive semidefinite almost everywhere on $\Tbb^d$.

In practice, we usually observe one finite-size realization of the underlying random field, from which we can estimate a {finite number of covariances. Let $\Lambda\subset\Zbb^d$ be the index set for these covariances. We shall require it to have a finite cardinality, contain the all-zero index, and be symmetric {with respect to} the origin, namely $\kb\in\Lambda$ implies $ -\kb\in\Lambda$. Hence $\Lambda$ must have an odd number of elements.} A well-established paradigm to estimate the spectrum of the random field is called \emph{covariance extension} which can be formalized as a \emph{trigonometric moment problem:} {given} the set of covariances $\{\Sigma_{\kb}:\,\kb\in\Lambda\}$ computed from the realization, find a spectral density $\Phi:\Tbb^d \to \Hfrak_{+,m}$ that solves the integral equations
\begin{equation}\label{moment_eqns}
\int_{\Tbb^d}e^{i\innerprod{\kb}{\thetab}}\Phi(e^{i\thetab})\d\m(\thetab) = \Sigma_{\kb}\text{ for all } \kb\in\Lambda,
\end{equation}
where {$\Hfrak_{+,m}$ is the cone of positive definite matrices of dimension $m$, and}
\begin{equation}
\d\m(\thetab)=\frac{1}{(2\pi)^d}\prod_{j=1}^{d} \d\theta_j
\end{equation}
is the normalized Lebesgue measure on $\Tbb^d$. In other words, we constrain the candidate solution to have {its} $\kb$-th Fourier coefficient equal to the given $\Sigma_{\kb}$.

It is well-known that the integral equations \eqref{moment_eqns} in general have infinitely many solutions if one {exists}. One way to {handle} such ill-posedness {problems} is to {use a regularization term}
\begin{equation}\label{primal}
\underset{\Phi\in\Sfrak_m}{\text{minimize}}\ D(\Phi,\Psi) \quad \text{subject to } \cref{moment_eqns},
\end{equation}
where we need to introduce two extra ingredients. One is a spectral density function $\Psi$ which represents our \emph{a priori} knowledge on the solution $\Phi$. The other ingredient is a cost functional $D$, very often a divergence index that measures the ``distance'' between two spectral densities. The set $\Sfrak_m$ contains bounded and coercive\footnote{A matricial spectral density $\Phi$ is bounded and coercive if there exist real numbers {$b>a>0$ such that $a I_m \leq \Phi(e^{i\thetab}) \leq b I_m$} for all $\theta\in\Tbb^d$.} $m\times m$ matricial spectral densities on $\Tbb^d$. Hence the idea is to seek a solution to the moment equations \eqref{moment_eqns} that is {the} closest to $\Psi$. {Such a} formulation is flexible since it allows solution selection by changing the prior $\Psi$. {In this multidimensional setting, however, the solution to \cref{primal} is not necessarily a spectral density, but rather a spectral measure that may contain a singular part.} In other words, the problem \cref{primal} posed for spectral densities may not be well defined. In order to tackle this, we will instead consider a similar formulation involving matricial measures, as will be detailed in the next section.

\section{Problem formulation}
\label{sec:problem}

In this paper, we work on the optimization problem \eqref{primal} having the {cost functional}
\begin{equation}\label{tau_diverg}
%\begin{split}
D_\tau(\Phi,\Psi) := \int_{\Tbb^d} \trace \left\{ \frac{1}{\tau(\tau-1)} (W_\Psi^{-1} \Phi W_\Psi^{-*})^\tau -\frac{1}{\tau-1} \Psi^{-1}\Phi \right\} \d\m + \frac{m}{\tau}
%\end{split}
\end{equation}
parametrized by a real variable {$\tau\in(0,1)$}{, termed \emph{$\tau$-divergence}}.
Here the function $W_\Psi$ is a \emph{pointwise} square root of the prior $\Psi$, that is, $\Psi(e^{i\thetab}) = W_\Psi(e^{i\thetab}) W^*_\Psi(e^{i\thetab})$ for $\thetab\in\Tbb^d$ {almost everywhere}\footnote{In unidimensional case ($d=1$), one can take $W_\Psi$ to be a \emph{spectral factor} of $\Psi$. However, this is in general not possible in the multidimensional setting.}. In particular, one can take $W_\Psi$ to be the pointwise Cholesky factor or the Hermitian square root. {Note that the value of \cref{tau_diverg} does not depend on the particular choice of the square root.} This family of divergence indices was introduced in \cite{Z-15}, where it was shown that $D_\tau$ is strictly convex in its first argument, it is a pseudo-distance, and the domain of $\tau$ can be extended to {$\tau=0$ and $\tau=1$} via continuity.
An important consequence is that $D_\tau$ connects the \emph{Itakura-Saito} distance and the modified\footnote{More precisely, the matricial Kullback-Leibler divergence between $W_\Psi^{-1} \Phi W_\Psi^{-*}$ and the constant identity matrix, cf.~\cite{Z-15}.} \emph{Kullback-Leibler} divergence in a continuous manner as $\tau$ ranges in the closed interval $[0,1]$.

We are interested in solving the optimization problem \eqref{primal} with $D_\tau$ such that $\tau=1-\frac{1}{\nu}$ for {$\nu\ge 2$ being} a positive integer. {This choice of the constant $\tau$ results in rational solutions in the scalar one-dimensional setting \cite{Z-14rat}.}
More explicitly, {the change of variable gives the objective functional}
\begin{equation}\label{D_nu}
%\begin{split}
D_{1-\frac{1}{\nu}}(\Phi,\Psi) = \int_{\Tbb^d} \trace \left\{ \frac{\nu^2}{1-\nu} (W_\Psi^{-1} \Phi W_\Psi^{-*})^{1-\frac{1}{\nu}} + \nu \Psi^{-1}\Phi \right\} \d\m + \frac{m\nu}{\nu-1}.
%\end{split}
\end{equation}
Next, we shall follow the idea of \cite{KLR-16multidimensional,RKL-16multidimensional} and consider {the setting where the function $\Phi$ in the integral is replaced by a {nonnegative matricial measure, i.e., a function defined on the Borel $\sigma$-algebra on $\Tbb^d$ which assigns a positive semidefinite matrix to each Borel set.}
The objective functional then becomes
\begin{equation}\label{primal_obj}
%\begin{split}
\Dcal_{1-\frac{1}{\nu}}(\d M,\Psi) := \int_{\Tbb^d} \trace \left\{ \frac{\nu^2}{1-\nu} (W_\Psi^{-1} \Phi W_\Psi^{-*})^{1-\frac{1}{\nu}}\d\m + \nu \Psi^{-1}\d M \right\} + \frac{m\nu}{\nu-1}
%\end{split}
\end{equation}
where $\Phi$ is the absolutely continuous part of the measure $\d M$, according to a matricial version of the Lebesgue decomposition theorem (see \cref{thm_Cramer_variant} in the Appendix)
\begin{equation}\label{decomp_Lebesgue}
\d M = \Phi\,\d\m + \d M_\s.
\end{equation}
Note that the singular part $\d M_\s$ does not appear in the first term in the integral in \eqref{primal_obj} since the exponent is less than $1$.} Clearly, $\Dcal_{1-\frac{1}{\nu}}(\d M,\Psi)$ can be viewed as a pseudo-distance between nonnegative matricial measures, i.e., its value is nonnegative and is equal to zero when $\d M=\Psi\d\m$. 
The optimization problem then becomes
\begin{equation}\label{primal_measure}
\begin{aligned}
& \underset{\d M \geq 0}{\text{minimize}}
& & \Dcal_{1-\frac{1}{\nu}}(\d M,\Psi) \\
& \text{subject to}
& & \int_{\Tbb^d}e^{i\innerprod{\kb}{\thetab}}\d M(\thetab) = \Sigma_{\kb}\text{ for all } \kb\in\Lambda,
\end{aligned}
\end{equation}
where $\d M$ is a nonnegative matricial measure with an absolutely continuous part $\Phi\d\m$. 

Since \eqref{primal_measure} is a constrained optimization problem, we shall assume that it is feasible.

\begin{assumption}[Feasibility]\label{assump_feasibility}
	There {exists} a nonnegative matricial measure $M_0$
	such that the equality constraints in \eqref{primal_measure} hold given those $\{\Sigma_{\kb}\}_{\kb\in\Lambda}$. Moreover, there {exists} a nonnegative scalar measure $\lambda$ such that $M_0$ has a Radon-Nikod\'{y}m derivative $\d M_0 = M'_{0,\lambda}\d\lambda$
	and the density $M'_{0,\lambda}$ is positive definite on some {open ball} $B\subset\Tbb^d$ {such that $\lambda(B)>0$}.\footnote{In order for the density to exist, each element of $M_0$ must be absolutely continuous {with respect to} $\lambda$. For example, one can take $\lambda(B)= \sum_{j,k} |(M_0)_{jk}|(B)$, sum of the total variation of each element of $M_0$. The density is positive semidefinite $\lambda$-a.e.~because $M_0$ is nonnegative.}
\end{assumption}

\begin{remark}
	{The assumption above is quite strong.}	{In \cite{geronimo2004positive,geronimo2005operator,kimsey2013truncated,geronimo2021autoregressive}, explicit conditions on the truncated covariances $\{\Sigma_{\kb}\}_{\kb\in\Lambda}$ have been formulated for the existence of a nonnegative representing measure which is even a spectral density of the autoregressive type in the $2$-d case. In general, however, the second part of the assumption (positivity of the density) is quite difficult to verify when the dimension is greater than two. 
	In practice, the feasibility can be guaranteed by a suitable covariance estimation scheme given a finite number of measurements. In particular, the estimated covariances correspond to the {smoothed} matricial periodogram, a spectral density (cf.~\cite[Section~V]{ZFKZ2019M2}). {It is also worth noting that}} if a coercive spectral density $\Phi_0$ solves the moment equations \eqref{moment_eqns}, {then} the measure $\Phi_0\d\m$ certainly meets the requirements in \cref{assump_feasibility}.
\end{remark}

The next assumption {puts} some constraints on the prior function $\Psi$, which we will need for the rest part of the paper.
\begin{assumption}\label{assump_Psi}
	Each element of the $m\times m$ matricial function $\Psi$ is a \emph{rational} function, i.e., ratio of two polynomials.
	Further, $\Psi$ is a bounded and coercive spectral density.
\end{assumption}
The rationality assumption is quite natural from the system-theoretic point of view. The second requirement rules out pathological cases since pole-zero cancellation may not be well defined for two polynomials of several variables.

\section{Duality analysis}\label{sec:duality}

In what follows, we will approach the optimization problem \eqref{primal_measure} via duality. Notice that the last term in the objective functional \eqref{primal_obj} is a constant, and {can be} ignored in the {analysis}. {Consider} the Lagrangian
\begin{equation}
\begin{split}
\Lcal_\nu(\d M,\Qb) & =\int_{\Tbb^d} \trace \left\{ \frac{\nu^2}{1-\nu} (W_\Psi^{-1} \Phi W_\Psi^{-*})^{1-\frac{1}{\nu}}\d\m + \nu \Psi^{-1}\d M \right\} \\
 & \quad + \sum_{\kb\in\Lambda} \trace \left[ Q_\kb \left(\int_{\Tbb^d} e^{i\innerprod{\kb}{\thetab}} \d M - \Sigma_{\kb}\right)^* \right] \\
 & = \int_{\Tbb^d} \trace \left\{ \frac{\nu^2}{1-\nu} (W_\Psi^{-1} \Phi W_\Psi^{-*})^{1-\frac{1}{\nu}}\d\m + (\nu \Psi^{-1} + Q) \d M \right\} - \innerprod{\Qb}{\Sigmab}
\end{split}
\end{equation}
where, the {Lagrange multipliers} $\Qb=\{Q_\kb\}_{\kb\in\Lambda}$, $Q_\kb\in\Cbb^{m \times m}$ {satisfies} $Q_{-\kb}=Q_\kb^*$, {and thus} $Q(e^{i\thetab}):=\sum_{\kb\in\Lambda} Q_\kb e^{-i\innerprod{\kb}{\thetab}}$ is a {Hermitian} matrix trigonometric polynomial of several variables. {Further, let} $\Sigmab=\{\Sigma_\kb\}_{\kb\in\Lambda}$ consist of the covariance data, and {denote} $\innerprod{\Qb}{\Sigmab}:=\sum_{\kb\in\Lambda}\trace(Q_\kb \Sigma_{\kb}^*)$.

For a fixed $\Qb$, consider the problem
\begin{equation}
\underset{\d M \geq 0}{\text{inf}}\ \Lcal_\nu(\d M,\Qb).
\end{equation}
The above infimum is finite only for those $\Qb$ in the set
\begin{equation}
%\begin{split}
\Lscr_+:=\left\{ \{Q_\kb\}_{\kb \in\Lambda} : \nu \Psi^{-1} + Q\geq 0 {\text{ on $\Tbb^d$ but not identically zero}}\right\}.
%\end{split}
\end{equation} 
We shall call $\Lscr_+$ feasible set. To see this fact, suppose first that the Hermitian matrix $(\nu \Psi^{-1} + Q)(e^{i\thetab_0})$ has a negative eigenvalue. We can write down its eigen-decomposition
\begin{equation}
(\nu \Psi^{-1} + Q)(e^{i\thetab_0}) = U(\thetab_0) {\it\Lambda}(\thetab_0) U(\thetab_0)^*,
\end{equation}
where ${\it\Lambda}(\thetab_0)=\diag\{\lambda_1(\thetab_0),\dots,\lambda_k(\thetab_0),\dots\lambda_m(\thetab_0)\}$ such that $\lambda_k(\thetab_0)<0$, and $U(\thetab_0)$ is a unitary matrix. We can take
\begin{equation}
\d M = U(\thetab_0)\, {\diag\{0,\dots,w,\dots,0\}}\, U(\thetab_0)^* \delta(\thetab-\thetab_0)\d\thetab,
\end{equation}
where {the real number $w>0$ appears on the $k$-th position of the diagonal matrix}, and $\delta(\cdot)$ is the Dirac delta. Now, the infimum of the Lagrangian is clearly determined by the part
\begin{equation}
\begin{split}
\trace \int_{\Tbb^d} (\nu \Psi^{-1} + Q) \d M = w \lambda_k(\thetab_0) 
\end{split}
\end{equation}
which tends to $-\infty$ as $w\to +\infty$.
On the other hand, if $\Qb$ is such that $\nu \Psi^{-1} + Q$ is identically zero, then we can choose $\d M = w\Psi\d\m$ with $w$ again a positive constant. This time, the dominant term in the Lagrangian is
\begin{equation}
\trace \int_{\Tbb^d} \frac{\nu^2}{1-\nu} (W_\Psi^{-1} \Phi W_\Psi^{-*})^{1-\frac{1}{\nu}}\d\m = \frac{\nu^2}{1-\nu} m w^{1-\frac{1}{\nu}}
\end{equation}
which also tends to $-\infty$ as $w\to+\infty$ because we have $\nu\geq2$. Therefore, we can restrict our attention to the feasible set $\Lscr_+$.

Using the Lebesgue decomposition \eqref{decomp_Lebesgue}, we see that the singular measure $\d M_\s$ appears in the Lagrangian only through the term
\begin{equation}\label{int_Ms}
\trace \int_{\Tbb^d} (\nu \Psi^{-1} + Q) \d M_\s = \trace \int_{\Tbb^d} (\nu \Psi^{-1} + Q) M'_{\s,\lambda} \d \lambda,
\end{equation}
where $\lambda$ here is a nonnegative (scalar) measure such that each element of $M_\s$ is absolutely continuous with respect to it\footnote{Such a measure $\lambda$ always exists. For example, it can be obtained by adding together the total variation of each element of $M_\s$.}, and $M'_{\s,\lambda}$ is a positive semidefinite matrix-valued (measurable) function, whose elements are the element-wise Radon-Nikod\'{y}m derivatives of $M_\s$ {with respect to} $\lambda$. Due to positive semidefiniteness, the integrand $\trace [(\nu \Psi^{-1} + Q) M'_{\s,\lambda}]\geq0$, and hence the integral in \eqref{int_Ms} is nonnegative. It follows that the Lagrangian
\begin{equation}\label{inequal_Lagrangian}
%\begin{split}
\Lcal_\nu(\d M,\Qb) \geq 
\int_{\Tbb^d} \trace \left\{ \frac{\nu^2}{1-\nu} (W_\Psi^{-1} \Phi W_\Psi^{-*})^{1-\frac{1}{\nu}} + (\nu \Psi^{-1} + Q) \Phi \right\}\d\m - \innerprod{\Qb}{\Sigmab},
%\end{split}
\end{equation}
which means that the infimum of $\Lcal_\nu(\d M,\Qb)$ can only be attained {at a measure} $\d M$ {for which} its singular part satisfies $\trace \int_{\Tbb^d} (\nu \Psi^{-1} + Q) \d M_\s = 0$. From \eqref{int_Ms}, we know that the consequence is $\trace [(\nu \Psi^{-1} + Q) M'_{\s,\lambda}]=0$ $\lambda$-{almost everywhere.} Recall {that} for two positive semidefinite matrices $A$ and $B$, $\trace(AB)=0$ {implies} $AB=0$. Hence the previous condition further implies that outside a $\lambda$-null set, whenever $\nu \Psi^{-1} + Q >0$ at some $e^{i\thetab}$, it must happen that $M'_{\s,\lambda}(e^{i\thetab})=0$. In other words, the support of $M'_{\s,\lambda}$ is contained in the zero set of the (trigonometric) rational function 
\begin{equation}\label{zero_set}
\Zcal(\Qb):=\left\{ \thetab\in\Tbb^d \,:\, \det\left[ \nu \Psi^{-1}(e^{i\thetab}) + Q(e^{i\thetab}) \right]=0  \right\},
\end{equation}
{where we have made explicit the dependence on $\Qb$.}

According to \cite{Z-15}, the functional on the right-hand side of \eqref{inequal_Lagrangian}, denoted as $\Lcal_\nu(\Phi\d\m,\Qb)$, is strictly convex in $\Phi$ in the set of bounded and coercive matricial spectral densities. In fact, the convexity can be extended to functions that are positive semidefinite {almost everywhere} although strict convexity is then lost.

\begin{lemma}\label{lem_convex_L_nu}
	{Let $L^1_{m\times m}$ denote the space of {$m\times m$ complex matrix-valued functions $F$ on $\Tbb^d$} such that the absolute value of each entry of $F$ is integrable}. For a fixed $\Qb\in\Lscr_+$, the functional $\Lcal_\nu(\Phi\d\m,\Qb)$ is convex in $\Phi$ over the set of Hermitian matrix-valued functions in $L^1_{m\times m}$ that are positive semidefinite {almost everywhere.}
\end{lemma}

\begin{proof}
	Since the last term of $\Lcal_\nu(\Phi\d\m,\Qb)$ is fixed and the second term is linear in $\Phi$, we only need to show the convexity of the first integral term. Since we are in the case $\nu\geq2$, it is equivalent to prove that the functional
	\begin{equation}
	g(\Phi) := \trace \int_{\Tbb^d} (W_\Psi^{-1} \Phi W_\Psi^{-*})^{1-\frac{1}{\nu}} \d\m
	\end{equation}
	is concave. The question then reduces to the concavity of the integrand. More precisely, fix a nonsingular matrix $A$, and define the function
	\begin{equation}
	f_A(B) := \trace\left\{ (A^{-1} B A^{-*})^{1-\frac{1}{\nu}} \right\}
	\end{equation}
	for $B\geq 0$. One can show that $f_A$ is strictly concave for $B>0$ via derivative-based analysis as given in \cite{Z-14}. Then the concavity can be extended to the boundary via continuity, namely to positive semidefinite matrices $B$. Finally, notice that
	\begin{equation}
	g(\Phi) = \trace \int_{\Tbb^d} f_{W_\Psi(\thetab)} (\Phi(\thetab)) \d\m,
	\end{equation}
	and concavity of $g$ follows using a pointwise argument on the integrand.
\end{proof}

The directional derivative of the Lagrangian in the $L^1_{m\times m}$ direction $\delta\Phi: \Tbb^d \to \Hfrak_{m}$, where $\Hfrak_{m}$ is the space of $m\times m$ Hermitian matrices, can be computed as
\begin{subequations}\label{derivative_Lagrangian}
\begin{align}
\delta \Lcal_\nu(\Phi\d\m,\Qb;\delta\Phi) & = \int_{\Tbb^d} \trace \left\{ -\nu (W_\Psi^{-1} \Phi W_\Psi^{-*})^{-\frac{1}{\nu}} W_\Psi^{-1} \delta\Phi W_\Psi^{-*} + (\nu \Psi^{-1} + Q) \delta\Phi \right\} \d\m \\
 & = \int_{\Tbb^d}\innerprod{-\nu W_\Psi^{-*} (W_\Psi^{-1} \Phi W_\Psi^{-*})^{-\frac{1}{\nu}} W_\Psi^{-1} + \nu \Psi^{-1} + Q}{\delta\Phi} \d\m \label{inner_prod_dPhi}
\end{align}
\end{subequations}
where we have used the fact that the directional derivative of $\trace(X^c)$ for $X>0$ and $c\in\Rbb$ is given by (cf.~\cite{Z-14})
\begin{equation}\label{diff_X^c}
\delta (\trace(X^c);\delta X) = c \trace(X^{c-1} \delta X).
\end{equation}
{I}n computing \eqref{derivative_Lagrangian}, we have interchanged the order of differentiation and integration. Such an operation can be justified using Lebesgue's dominated convergence theorem \emph{if $\Phi$ is coercive}.  In that case, we can impose the directional derivative $\delta \Lcal_\nu(\Phi\d\m,\Qb;\delta\Phi)$ to vanish in any feasible direction $\delta\Phi\in L^1_{m\times m}$ such that $\Phi+\varepsilon\,\delta\Phi$ is nonnegative {almost everywhere} for sufficiently small $\varepsilon>0$. In particular, taking $\delta\Phi$ equal to the first member of the inner product in \eqref{inner_prod_dPhi} yields
\begin{equation}
\nu W_\Psi^{-*} (W_\Psi^{-1} \Phi W_\Psi^{-*})^{-\frac{1}{\nu}} W_\Psi^{-1} = \nu \Psi^{-1} + Q\quad {\text{a.e.}}
\end{equation}
After some calculation, we recover the stationary point
\begin{equation}\label{Phi_nu}
\begin{split}
\Phi_\nu & := \Psi \left[ (\Psi^{-1}+\frac{1}{\nu}Q) \Psi \right]^{-\nu} \\
& = (\Psi^{-1}+\frac{1}{\nu}Q)^{-1} \, \Psi^{-1} \cdots \Psi^{-1} \, (\Psi^{-1}+\frac{1}{\nu}Q)^{-1}
\end{split}
\end{equation}
where the last expression contains $\nu$ copies of $(\Psi^{-1}+\frac{1}{\nu}Q)^{-1}$. Apparently, such $\Phi_\nu$ is coercive, and a posteriori, the functional $\Lcal_\nu(\Phi\d\m,\Qb)$ is indeed differentiable at $\Phi_\nu$. Moreover, {as} a stationary point, {$\Phi_\nu$ must be} a minimizer of $\Lcal_\nu(\Phi\d\m,\Qb)$ {by the convexity property} (\cref{lem_convex_L_nu}).

Insert $\Phi_\nu$ into the Lagrangian, and we get the dual problem {of maximizing}
\begin{equation}
\Lcal_\nu(\Phi_\nu\d\m,\Qb) = \int_{\Tbb^d} \trace \left\{ \frac{\nu}{1-\nu} \left[ \Psi^{-1} (\Psi^{-1} + \frac{1}{\nu} Q)^{-1} \right]^{\nu-1} \right\} \d\m - \innerprod{\Qb}{\Sigmab}
\end{equation}
{with respect to} $\Qb\in\Lscr_+$.
Clearly, it is equivalent to consider the problem
\begin{equation}\label{J_dual}
%\begin{split}
\underset{\Qb\in\Lscr_+}{\text{minimize}} \quad J_\nu(\Qb) := - \Lcal_\nu(\Phi_\nu\d\m,\Qb)
%\end{split}
\end{equation}
and we will call
\begin{equation}\label{J_dual_func}
%\begin{split}
J_\nu(\Qb) = \langle\Qb,\Sigmab\rangle + \frac{\nu}{\nu-1} \int_{\Tbb^d} \trace\left\{ \left[ \Psi^{-1} (\Psi^{-1} + \frac{1}{\nu} Q)^{-1} \right]^{\nu-1} \right\} \d\m
%\end{split}
\end{equation}
the dual function.
An immediate comment is that the dual function does not depend on the factor of $\Psi$ that appears in the primal problem.

\section{Solution to the dual problem: existence and uniqueness}\label{sec:solution_dual}

Since $\Psi$ is rational by \cref{assump_Psi}, we can extend the dual function to the boundary of the feasible set, denoted as $\partial\Lscr_+$. More precisely, $\Qb\in\partial\Lscr_+$ if the matrix function $\Psi^{-1}+\frac{1}{\nu} Q$ is positive semidefinite on $\Tbb^d$ and is singular at some $e^{i\thetab}$. The function $\Psi(\zb)^{-1}+\frac{1}{\nu} Q(\zb)$ is rational in $\zb\in\Cbb^d$, and so is its determinant. Therefore, {the {corresponding} $\Zcal(\Qb)$ in \eqref{zero_set} is just the zero set of a multivariate trigonometric polynomial which is well known to have Lebesgue measure zero in $\Tbb^d$}. {We conclude that the} function $J_\nu$ at $\Qb\in\partial\Lscr_+$ admits the expression in \eqref{J_dual_func}, while the domain of integration is now changed to exclude those zero points.

Let us compute the first variation of the dual function 
\begin{equation}\label{delta_J_nu}
\begin{split}
\delta J_\nu(\Qb;\delta\Qb) & = \langle\delta\Qb,\Sigmab\rangle - \int_{\Tbb^d} \trace\left\{ \Psi\left[ (\Psi^{-1} + \frac{1}{\nu} Q) \Psi \right]^{-\nu} \delta Q \right\} \d\m \\
 & = \trace \left\{ \sum_{\kb\in\Lambda} \delta Q_\kb \left( \Sigma_{\kb}^* - \int_{\Tbb^d} e^{-i\innerprod{\kb}{\thetab}} \Psi\left[ (\Psi^{-1} + \frac{1}{\nu} Q) \Psi \right]^{-\nu} \d\m \right) \right\}
\end{split}
\end{equation}
{where $\Qb\in\interior(\Lscr_+)$, i.e., the interior of $\Lscr_+$ such that $\Psi^{-1}+\frac{1}{\nu}Q>0$ on {$\Tbb^d$}. As a consequence,} we can take the differential inside the integral. In that case, the stationarity condition is
\begin{equation}\label{stationary_cond}
\int_{\Tbb^d} e^{i\innerprod{\kb}{\thetab}} \Psi \left[ (\Psi^{-1}+\frac{1}{\nu}Q) \Psi \right]^{-\nu} \d\m = \Sigma_\kb\quad\forall \kb\in\Lambda.
\end{equation}
That is to say, if the optimal $\Qb^\circ$ of the dual {problem} lies in $\interior(\Lscr_+)$, then the corresponding $\Phi_\nu(\Qb^\circ)$ in \eqref{Phi_nu} solves the moment equations \eqref{moment_eqns}. It then follows from the primal-dual complementarity that $\Phi_\nu(\Qb^\circ)\d\m$ is also a solution to the primal problem \eqref{primal_measure}. The remaining part of this section will be devoted to the proof of the next result.

\begin{theorem}\label{thm_dual}
	If $\cref{assump_feasibility}$ holds, then the dual problem \eqref{J_dual} admits a unique solution in the closed set $\Lscr_+$. 
\end{theorem}

We will break down the proof into two subsections: {the former shows the uniqueness of the solution, if it exists; the latter regards its existence.}

\subsection{Uniqueness}

We just need to prove that the dual function is strictly convex in $\Lscr_+$. The uniqueness of the solution then follows if a solution \emph{exists}. The idea is to compute the second variation of $J_\nu$ in $\interior(\Lscr_+)$, and show its positive definiteness. Moreover, we show that strict convexity holds even when the boundary of $\Lscr_+$ is taken into consideration.

\begin{lemma}\label{lem_f_convex}
	Given a positive definite matrix $X$ and an integer $n\geq1$, the function $f_{X,n}(Y):=\trace\{(XY^{-1})^n\}$ is strictly convex in $Y>0$.
\end{lemma}

\begin{proof}
	The function $f_{X,n}$ is clearly smooth, so the aim is to show that the second differential is positive definite. The computation is similar to those in the proof of \cite[Theorem~5.1]{Z-14}, and hence is omitted here.
\end{proof}

\begin{proposition}\label{prop_strict_convex}
	The dual function $J_\nu(\Qb)$ is strictly convex in $\Lscr_+$.
\end{proposition}
\begin{proof}
	Since the first term of $J_\nu$ in \eqref{J_dual_func} is linear in $\Qb$ and the constant $\frac{\nu}{\nu-1}>0$, we only need to show the strict convexity of the integral term
	\begin{equation}\label{g_Q}
	%\begin{split}
	g(\Qb) :=\int_{\Tbb^d} \trace\left\{ \left[ \Psi^{-1} (\Psi^{-1} + \frac{1}{\nu} Q)^{-1} \right]^{\nu-1} \right\} \d\m.
	%\end{split}
	\end{equation}
	To ease the notation, let us write $\Psi$ and $Q$ as functions of $\thetab$, and define $R_\Qb:=\Psi^{-1} + \frac{1}{\nu} Q$. Then we have $g(\Qb)=\int_{\Tbb^d} f_{\Psi^{-1}(\thetab),\nu-1}(R_\Qb(\thetab)) \d\m${, where the function $f$ has been defined in the previous lemma}.
	We shall do {a} pointwise reasoning with the integrand and show convexity according to the definition. 
	
	Let $\Qb_1,\Qb_2 \in \Lscr_+$ be two different points, and for $t\in(0,1)$, we have $R_{t\Qb_1+(1-t)\Qb_2} = t R_{\Qb_1} + (1-t)R_{\Qb_2}$. {Notice that $R_{\Qb_1}\neq R_{\Qb_1}$.} Consider
	\begin{equation}
	\begin{split}
	& g(t\Qb_1+(1-t)\Qb_2) \\
	= & \int_{\Tbb^d} f_{\Psi^{-1}(\thetab),\nu-1}\left(tR_{\Qb_1}(\thetab) + (1-t)R_{\Qb_2}(\thetab)\right) \d\m \\
	< & \int_{\Tbb^d} \left[ t f_{\Psi^{-1}(\thetab),\nu-1}(R_{\Qb_1}(\thetab)) + (1-t) f_{\Psi^{-1}(\thetab),\nu-1}(R_{\Qb_2}(\thetab)) \right] \d\m \\
	= & tg(\Qb_1) + (1-t)g(\Qb_2),
	\end{split}
	\end{equation}
	where the inequality follows from the strict convexity of the integrand $f_{\Psi^{-1}(\thetab),\nu-1}$ by \cref{lem_f_convex}. Notice that the above reasoning still holds for $\Qb_1$ or $\Qb_2$ in $\partial\Lscr_+$, because we only need to {change} the domain of integration {to exclude the zero sets $\Zcal(t\Qb_1+(1-t)\Qb_2)$, $\Zcal(\Qb_1)$, and $\Zcal(\Qb_2)$ as defined in \cref{zero_set}.}
\end{proof}

%(implies uniqueness of the solution to the dual problem)

\subsection{Existence}\label{subsec:exist}

As we will see next, the existence proof relies heavily on the feasibility \cref{assump_feasibility}.
We start by showing that the dual function $J_\nu$ is bounded from below on $\Lscr_+$, for which we need the following lemma.

\begin{lemma}\label{lem_innerprod_inequal}
	If \cref{assump_feasibility} holds, then there exists a real number $\alpha$ such that for any $\Qb\in\Lscr_+\,$, the inequality
	$\innerprod{\Qb}{\Sigmab} \geq \alpha$
	holds.
\end{lemma}

\begin{proof}
    Given the nonnegative measure $M_0$ in the feasibility assumption, we can do the following computation:
	\begin{align}
	%\begin{split}
	\innerprod{\Qb}{\Sigmab}& := \sum_{\kb\in\Lambda} \trace (Q_\kb\Sigma_\kb^*) \nonumber \\ 
	& = \sum_{\kb\in\Lambda} \trace \left( Q_\kb \int_{\Tbb^d} e^{-i\innerprod{\kb}{\thetab}} \d M_0 \right) \nonumber \\
	& = \trace \int_{\Tbb^d} Q \, \d M_0 \nonumber \\
	& = \nu \trace \int_{\Tbb^d} (\Psi^{-1} + \frac{1}{\nu}Q - \Psi^{-1}) \d M_0 \nonumber \\
	& = \nu \left[ \trace \int_{\Tbb^d} (\Psi^{-1} + \frac{1}{\nu}Q) \d M_0 - \trace \int_{\Tbb^d} \Psi^{-1} \d M_0 \right]. \label{innerprod_inequal_deriv}
	%\end{split}
	\end{align}
	Since $\Qb\in\Lscr_+$, the first term in
	\eqref{innerprod_inequal_deriv} is nonnegative. Therefore, we have  
	\begin{equation}
	\innerprod{\Qb}{\Sigmab}\geq -\nu \trace \int_{\Tbb^d} \Psi^{-1} \d M_0 =:\alpha.
	\end{equation}
\end{proof}

An immediate consequence is that for any $\Qb\in\Lscr_+$, the function value $J_\nu(\Qb)$ is bounded from below. To see this, just notice that the second term of $J_\nu$ in \eqref{J_dual_func} is nonnegative since the constant $\nu\geq2$. Therefore, we have
\begin{equation}\label{inequal_J_nu}
J_\nu(\Qb)\geq\innerprod{\Qb}{\Sigmab}\geq\alpha.
\end{equation}
In particular, this implies that the minimum of the dual function on $\Lscr_+$ is not $-\infty$.

\begin{lemma}\label{lem_lower_semicon}
	The dual function $J_\nu(\Qb)$ is lower-semicontinuous on $\Lscr_+$.
\end{lemma}
\begin{proof}
	Derivative-based analysis can be carried out to show that $J_\nu$ is smooth on $\interior(\Lscr_+)$, and thus of course continuous. We only need to show the lower-semicontinuity for any $\bar{\Qb}\in\partial\Lscr_+$. More precisely, since the first term of $J_\nu$ is continuous, it is sufficient to show that the function $g(\Qb)$ defined in \eqref{g_Q} is lower-semicontinuous.
	
	We shall refer to the notation used in the proof of \cref{prop_strict_convex}. Let $\{\Qb_k\}_{k\geq1}\subset\Lscr_+$ be a sequence that converges to $\bar{\Qb}$ on the boundary. For almost every $\thetab$, the integrands
	$f_{\Psi^{-1}(\thetab),\nu-1}(R_{\Qb_k}(\thetab))$ $(k=1,2,\dots)$ are well defined and nonnegative.\footnote{Those $\thetab$ such that $\det R_{\Qb_k}(\thetab)=0$ $(k=1,2,\dots)$, which form a set of Lebesgue measure zero, are excluded.} Moreover, the pointwise limit
	\begin{equation}
	\lim_{k\to\infty} f_{\Psi^{-1}(\thetab),\nu-1}(R_{\Qb_k}(\thetab)) = f_{\Psi^{-1}(\thetab),\nu-1}(R_{\bar{\Qb}}(\thetab))
	\end{equation}
	holds {almost everywhere.}	
	By Fatou's lemma \cite[p.~23]{rudin1987real}, we have
	\begin{equation}
	g(\bar{\Qb}) \leq \liminf_{k\to\infty} g(\Qb_k).
	\end{equation}
	Since $\{\Qb_k\}_{k\geq1}$ is an arbitrary sequence tending to $\bar{\Qb}$, we have proved the lower-semicontinuity of the function $g$ at $\bar{\Qb}$, and the assertion of the lemma follows.
\end{proof}

%(implies the closedness of sublevel sets)

Take a sufficiently large real number $r$, and define the (nonempty) sublevel set of the dual function:
\begin{equation}
J_\nu^{-1}(-\infty,r] := \{ \Qb\in\Lscr_+ : J_\nu(\Qb)\leq r \}.
\end{equation}
\cref{lem_lower_semicon} then implies that the sublevel set is closed.

For the next lemma, let us define the norm of the Lagrange multiplier 
\begin{equation}\label{Q_norm}
\|\Qb\|:= \sqrt{ \sum_{\kb\in\Lambda} \trace (Q_\kb Q^*_\kb) }.
\end{equation}

%proof of function value escapes to infinity when $\|\Qb\|\to\infty$ 

\begin{lemma}\label{lem_unbounded_Q}
	If a sequence $\{\Qb_k\}_{k\geq1}\subset \Lscr_+$ is such that $\|\Qb_k\|\to\infty$ as $k\to\infty$, then 
	\begin{equation}\label{limit_inf_Q}
	\lim_{k\to\infty} J_\nu(\Qb_{k}) = \infty.
	\end{equation}
\end{lemma}

\begin{proof}
	Due to the relation \eqref{inequal_J_nu}, it suffices to prove the statement of the lemma for the inner product $\innerprod{\Qb}{\Sigmab}$. 
	Given the sequence {$\{\Qb_k\}_{k\geq1}$}, define $\Qb_k^0:=\Qb_k / \|\Qb_k\|$, which necessarily implies that $Q_k^0(e^{i\thetab})=Q_k(e^{i\thetab}) / \|\Qb_k\|$. Moreover, {for} each $\Qb_k\in {\Lscr_+}$, we have $\Psi^{-1}+\frac{1}{\nu}Q_k \geq 0$ on $\Tbb^d$. Consequently, the function
	\begin{equation}
	\Psi^{-1} + \frac{1}{\nu} Q_k^0 = \frac{1}{\|\Qb_k\|} (\Psi^{-1}+\frac{1}{\nu}Q_k) + \left( 1 - \frac{1}{\|\Qb_k\|} \right) \Psi^{-1}
	\end{equation}
	is positive definite on $\Tbb^d$ for sufficiently large $k$ since $\|\Qb_k\|\to\infty$. To summarize, the sequence $\{\Qb_k^0\}_{k\geq1}$ lives on the unit surface $\|\Qb\|=1$ (a compact set due to finite dimensionality), and we have $\Qb_k^0\in\Lscr_+$ for $k$ large enough.

	From {\cref{inequal_J_nu}}, we have
	\begin{equation}
	\innerprod{\Qb_k^0}{\Sigmab} = \frac{1}{\|\Qb_k\|} \innerprod{\Qb_k}{\Sigmab} \geq \frac{\alpha}{\|\Qb_k\|} \to 0.
	\end{equation}
	Define the real quantity $\eta:= \liminf_{k\to\infty} \innerprod{\Qb_k^0}{\Sigmab}$. Then it must hold that $\eta\geq 0$. By a property of the limit inferior, we know that $\{\Qb_k^0\}_{k\geq1}$ has a subsequence $\{\Qb^0_{k_\ell}\}_{\ell\geq1}$ such that $\innerprod{\Qb_{k_\ell}^0}{\Sigmab}\to\eta$ as $\ell\to\infty$. Since $\{\Qb^0_{k_\ell}\}_{\ell\geq 1}$ is contained on the unit surface, it has a convergent subsequence denoted by $\{\Qb^0_{k_j}\}_{j\geq 1}$. Define the limit
	\begin{equation}
	\Qb^0_{\infty}:=\lim_{j\to\infty}\Qb_{k_j}^0.
	\end{equation}
	Then by the continuity of the inner product, we have $\eta=\innerprod{\Qb^0_{\infty}}{\Sigmab}$.
	
	Next, we show that $\Qb^0_\infty \in \interior\Lscr_+$. Since $\Qb_k\in {\Lscr_+}$, it holds that $\Psi^{-1}+\frac{1}{\nu}Q_k\geq0$ on $\Tbb^d$ for all $k$. This implies that
	\begin{equation}
	\frac{\Psi^{-1}}{\|\Qb_{k_j}\|} + \frac{1}{\nu} Q^0_{k_j} \geq 0 \text{ on } \Tbb^d \quad \forall j.
	\end{equation}
	The function on the left side of the above inequality converges uniformly to the polynomial $\frac{1}{\nu} Q^0_\infty$. Hence we must have $\frac{1}{\nu} Q^0_\infty\geq 0$ on $\Tbb^d$. As a consequence, $\Psi^{-1} + \frac{1}{\nu} Q^0_\infty>0$ on $\Tbb^d$ and indeed $\Qb^0_\infty \in\interior \Lscr_+$.
	
	The next step is to prove that $\eta=\innerprod{\Qb^0_{\infty}}{\Sigmab}>0$. Following the computation in the proof of \cref{lem_innerprod_inequal}, we arrive at
	\begin{equation}\label{eq:Q0inf}
	\begin{split}
	\innerprod{\Qb^0_\infty}{\Sigmab} & = \trace \int_{\Tbb^d} Q^0_\infty \d M_0 \\
	& = \int_{\Tbb^d} \trace(Q^0_\infty M'_{0,\lambda}) \d\lambda.
	\end{split}
	\end{equation}
    Since we have just proved that $Q^0_\infty$ is positive semidefinite on $\Tbb^d$, the integrand above takes nonnegative real values. Thus, $\innerprod{\Qb^0_{\infty}}{\Sigmab}=0$ implies that $\trace(Q^0_\infty M'_{0,\lambda})=0$ $\lambda$-{almost everywhere}, which gives $Q^0_\infty M'_{0,\lambda}=0$ due to positive semidefiniteness. By the second part of \cref{assump_feasibility}, on the open ball $B\subset\Tbb^d$, the polynomial $Q^0_\infty(e^{i\thetab})$ vanishes identically. By \cite[Lemma~1]{ringh2015multidimensional}, we must have $\Qb^0_\infty=\zerob$, which is a contradiction since we also have $\|\Qb^0_\infty\|=1$. Therefore, it must hold that $\eta>0$.
	
	Finally, since {$\eta = \liminf_{k\to\infty} \innerprod{\Qb_k^0}{\Sigmab}$} by definition, there exists an integer $k>0$ such that {$\innerprod{\Qb^0_{j}}{\Sigmab} > \eta/2$ for all $j> k$}. Then for {$j>k$ we have} 
	\begin{equation}
	\begin{split}
	J_\nu(\Qb_{j}) & \geq  \innerprod{\Qb_{j}}{\Sigmab} \\
	& = \|\Qb_{j}\| \innerprod{\Qb_{j}^0}{\Sigmab} \\
	& \geq \frac{\eta}{2} \|\Qb_{j}\|
	\end{split}
	\end{equation}
    which tends to infinity as $j\to\infty$.
\end{proof}

As a consequence of \cref{lem_unbounded_Q}, the sublevel set $J_\nu^{-1}(-\infty,r]$ has to be bounded. Recall that the dual variable {belongs to} a finite-dimensional space. Therefore, $J_\nu^{-1}(-\infty,r]$ is a compact subset of $\Lscr_+$. Putting {these} pieces together, we have a lower-semicontinuous function $J_\nu$ whose sublevel set is compact. By the extreme value theorem, the function $J_\nu$ attains its minimum over $\Lscr_+$. This concludes the existence proof.

%(implies the boundedness of sublevel sets), 

\section{An integrability condition and the interior solution}
\label{sec:integrability}

\cref{thm_dual} in the previous section does not exclude the possibility that the optimal $\Qb^\circ$ may fall on the boundary $\partial\Lscr_+$, in which case it is not necessarily a stationary point and the corresponding primal ``variable'' $\Phi_\nu(\Qb^\circ)$ in \eqref{Phi_nu} may not satisfy the moment equations. In other words, the measure $\Phi_\nu(\Qb^\circ)\d\m$ may be primal infeasible. {The aim of this section is to show that the parameter $\nu$, parametrizing the divergence family, can cure such primal infeasibility. More precisely, choosing $\nu$ sufficiently large,} the existence of an interior minimizer $\Qb^\circ$ {is guaranteed}. In that case, we can conclude the primal optimality of the absolutely continuous measure $\Phi_\nu(\Qb^\circ)\d\m$ with a rational coercive density.

\begin{proposition}\label{prop_int_blowup}
If $\nu\geq\frac{md}{2}+1$, then $J_\nu(\bar{\Qb})=\infty$ for $\bar{\Qb}$ on the boundary {$\partial\Lscr_+$} of the feasible set.
\end{proposition}

\begin{proof}
	Let us rename $R:=\Psi^{-1}+\frac{1}{\nu}\bar{Q}$. Since the term $\langle\bar{\Qb},\Sigmab\rangle$ is finite and the scalar $\frac{\nu}{\nu-1}>0$, it suffices to establish that the integral {in \cref{J_dual_func}} {diverges}. Moreover, the integrand is nonnegative, and hence we can restrict our attention to a {closed} neighborhood $N(\thetab_0)$ of some $\thetab_0$ where $\det R(e^{i\thetab_0})=0$. By the trace inequality in \cite{bushell1990trace}, we have
	\begin{equation}\label{inequal_part1}
	\int_{N(\thetab_0)} \trace\left[ ( \Psi^{-1} R^{-1} )^{\nu-1} \right] \d\m \geq \beta \int_{N(\thetab_0)} \trace \left[ R^{-(\nu-1)} \right] \d\m
	\end{equation}
	for some constant $\beta>0$ since the eigenvalues of $\Psi$ are bounded. Continuing \eqref{inequal_part1}, we have
	\begin{equation}\label{inequal_part2}
	\begin{split}
	\int_{N(\thetab_0)} \trace \left[ R^{-(\nu-1)} \right] \d\m & = \int_{N(\thetab_0)} \trace \left[ \left(\frac{\adjugate R}{\det R}\right)^{\nu-1} \right] \d\m \\
	 & = \int_{N(\thetab_0)} \frac{1}{(\det R)^{\nu-1}} \trace \left[ \left(\adjugate R\right)^{\nu-1} \right] \d\m \\
	 & \geq \int_{N(\thetab_0)} \frac{m}{(\det R)^{\nu-1}} \left[ \det \left(\adjugate R\right) \right]^{\frac{\nu-1}{m}} \d\m
	\end{split}
	\end{equation}
	where $\adjugate$ denotes the adjugate matrix, and we have used \cref{lem_tr_det} in the {Appendix} for the last inequality. Using the relation $\det(\adjugate A)=(\det A)^{m-1}$ for any square $m\times m$ matrix $A$, we can simplify the last expression to obtain
	\begin{equation}\label{inequal_part3}
	%\begin{split}
	\int_{N(\thetab_0)} \trace \left[ R^{-(\nu-1)} \right] \d\m \geq m \int_{N(\thetab_0)} \left(\det R\right)^{-\frac{\nu-1}{m}} \d\m.
	%\end{split}
	\end{equation}
	Given {\cref{assump_Psi}} on the prior $\Psi$, we can write $\Psi=\frac{N_\Psi}{d_\Psi}$, where $d_\Psi$ is a product of all the denominators of $\Psi_{jk}$ (element of $\Psi(\zb)$ at $(j,k)$ position, $j,k=1,\dots,m$) and $N_\Psi$ is a matrix of polynomials. It follows that $\det N_\Psi(e^{i\thetab})\neq 0$ for all $\thetab\in\Tbb^d$. Back to \eqref{inequal_part3}, we have
	\begin{equation}
	\begin{split}
	R = \Psi^{-1}+\frac{1}{\nu}\bar{Q} & = \frac{d_\Psi}{\det N_\Psi} \adjugate N_\Psi + \frac{1}{\nu}\bar{Q} \\
	 & = \frac{1}{\det N_\Psi} \left( d_\Psi \adjugate N_\Psi + \frac{\det N_\Psi}{\nu}\bar{Q} \right) := \frac{N_R}{d_R},
	\end{split}
	\end{equation}
	where $N_R$ is certainly a matrix polynomial, and $\det N_R(e^{i\thetab_0})=0$. In particular, we can always make $d_R(e^{i\thetab_0})>0$ in a neighborhood of $\thetab_0$ so that $N_R$ is positive semidefinite. Let $\nu-1=m\tilde{\nu}$ for some positive integer $\tilde{\nu}$. Now we can continue \eqref{inequal_part3}:
	\begin{equation}\label{inequal_part4}
	\begin{split}
	\int_{N(\thetab_0)} \trace \left[ R^{-(\nu-1)} \right] \d\m & \geq m \int_{N(\thetab_0)} \left(\det R\right)^{-\tilde{\nu}} \d\m \\
     & = m \int_{N(\thetab_0)} \frac{d_R^{m\tilde{\nu}}}{(\det N_R)^{\tilde{\nu}}} \d\m \\
	 & \geq m d_{\textrm{min}}^{m\tilde{\nu}} \int_{N(\thetab_0)} \left(\det N_R\right)^{-\tilde{\nu}} \d\m, \\
	\end{split}
	\end{equation}
	where $d_{\textrm{min}}:=\min_{\thetab\in N(\thetab_0)} d_R(\thetab)$ is a positive constant.
	By \cref{prop_Johan} in the {Appendix}, the last integral is unbounded if $\tilde{\nu}\geq\frac{d}{2}$ which is the same as $\nu\geq\frac{md}{2}+1$.
\end{proof}

It then follows that we can \emph{always} choose an integer $\nu\geq\frac{md}{2}+1$, such that the optimal dual variable $\Qb^\circ$ lies in the interior of $\Lscr_+$, and the spectral density $\Phi_\nu(\Qb^\circ)$ solves the moment equations.

\begin{remark}
	The above bound for $\nu$ is not tight in the unidimensional case. Indeed, letting $d=1$ and $\nu=2$, the inequality for $\nu$ implies that $m\leq2$. However, according to \cite{Z-15}, there is no such restriction for the number of variables in the $1$-d case.	
	In fact, the unidimensional case is very special because one can reason directly with the integrand $r:=\trace \left[ R^{-(\nu-1)} \right]$ in \eqref{inequal_part2} without {making the restriction $\nu=1+m\tilde{\nu}$ for an integer $\tilde{\nu}$}. The function $r(z)$ is a rational function of one variable, and $r(e^{i\theta})\to\infty$ as $\theta\to\theta_0$, which means that $e^{i\theta_0}$ must be a pole, and hence the integral necessarily blows up.
	This type of reasoning does not seem to extend to the multidimensional case.
\end{remark}

\section{Concerning the singular measure}\label{sec:measure}

When the condition for $\nu$ in \cref{prop_int_blowup} is not met, then it is not guaranteed that the dual problem has an interior solution. In that case, we need to add a singular measure to the absolutely continuous part $\Phi_\nu(\Qb^\circ)\d\m$ in order to achieve primal {feasibility}. The main technical tool here is Theorem 25.6 in \cite{rockafellar1970convex} which gives a characterization of the subdifferential of a differentiable convex function.

Assume that $\Qb^\circ \in \partial\Lscr_+$ is the unique minimizer of $J_\nu(\Qb)$. Then the all-zero vector $\zerob$ belongs to the subdifferential of $J_\nu$ at $\Qb^\circ$, denoted with $\partial J_\nu (\Qb^\circ)$. Since we have shown in \cref{subsec:exist} that the dual function is lower-semicontinuous, bounded from below, and its domain has a nonempty interior, according to \cite[Theorem~25.6]{rockafellar1970convex}, its subdifferential admits a decomposition
\begin{equation}\label{thm_Rockafellar}
\partial J_\nu (\Qb^\circ) = \mathrm{cl} \, (\mathrm{conv} \, S(\Qb^\circ)) + K(\Qb^\circ),
\end{equation}
where $\mathrm{cl} \, (\mathrm{conv} \, \cdot \,)$ denotes the closure of the convex hull of a set, $S(\Qb^\circ)$ is the set of all limit points of sequences of the form $\nabla J_\nu(\Qb_1), \nabla J_\nu(\Qb_2),\dots$ such that $\Qb_\ell\in\interior(\Lscr_+)$ and $\Qb_\ell$ tends to $\Qb^\circ$, and $K(\Qb^\circ):=\{\Sigmab_K \,:\, \innerprod{\Sigmab_K}{\Qb-\Qb^\circ} \leq 0 \textrm{ for all } \Qb\in\Lscr_+\}$ is the normal cone.

\subsection{Characterization of $S(\Qb^\circ)$}
In order to compute the gradient of $J_\nu$, we need to fix a basis for the finite-dimensional object $\Qb$. More precisely, let $\{\Xb_j\}_{j=1}^{N}$ be an orthonormal basis, so that we can write $\Qb=\sum_{j=1}^{N} q_j \Xb_j$ with real coordinates $q_j$. With a slight abuse of notation, we can regard $J_\nu$ as a function of the coordinate vector $\qb$.
To ease the notation, let us also define the linear operator that sends a Hermitian matricial measure on $\Tbb^d$ to its Fourier coefficients with indices in the set $\Lambda$
\begin{equation}
\Gamma:\, \d M \mapsto \left\{ \Sigma_{\kb}=\int_{\Tbb^d} e^{i\innerprod{\kb}{\thetab}} \, \d M \right\}_{\kb\in\Lambda}.
\end{equation}
Then according to \eqref{delta_J_nu}, the partial derivative can be expressed as
\begin{equation}
\frac{\partial J_\nu(\qb)}{\partial q_j} = \delta J_\nu(\qb;\Xb_j) = \innerprod{\Xb_j}{\Sigmab - \Gamma(\Phi_\nu(\Qb)\d\mu)}.
\end{equation}

Now take a vector $\vb \in S(\Qb^\circ)$. Then there exists a sequence $\{\qb_k\}_{k\geq1}\subset\interior(\Lscr_+)$ such that $\qb_k\to\qb^\circ$ and the gradient sequence $\{\nabla J_\nu(\qb_k)\}_{k\geq1}$ converges to $\vb$. This necessarily implies that the sequence of moments $\Gamma(\Phi_\nu(\Qb_k)\d\mu)$ converges as $k\to \infty$. In particular, convergence of the zeroth moments $\int_{\Tbb^d}\Phi_\nu(\Qb_k)\d\mu$ means that the matricial total variations of the sequence of measures $\Phi_\nu(\Qb_k)\d\mu$ are bounded. Identify each matricial measure as a linear functional on the space of Hermitian matrix-valued continuous functions in the sense of \cref{Riesz_represent} in the Appendix. Then by the Banach-Alaoglu theorem, there is a subsequence of $\Phi_\nu(\Qb_k)\d\mu$ that converges in weak* to some Hermitian measure. {Clearly,} the rational functions $\Phi_\nu(\Qb_k)$ converge uniformly to $\Phi_\nu(\Qb^\circ)$ in any compact subset of $\Tbb^d\backslash\Zcal(\Qb^\circ)$ {where the zero set has been defined in \cref{zero_set}. Therefore,} the weak* limit must have the form $\Phi_\nu(\Qb^\circ)\d\mu + \d M_S$ where $M_S$ is a nonnegative matricial measure satisfying
\begin{equation}\label{constraint_M_S}
\trace \int_{\Tbb^d} (\nu \Psi^{-1} + Q^\circ) \d M_S = 0.
\end{equation}
In other words, $M_S$ is only supported in the zero set $\Zcal(\Qb^\circ)$, and more specifically, in the kernel of the matrix $(\nu \Psi^{-1} + Q^\circ)(e^{i\thetab})$. To summarize, we have the relation
\begin{equation}
%\begin{split}
\vb \in \left\{\ub \,:\, u_j = \innerprod{\Xb_j}{\Sigmab - \Gamma(\Phi_\nu(\Qb^\circ)\d\mu + \d M_S)} \textrm{ with } M_S\geq0 \textrm{ and satisfies \eqref{constraint_M_S}}\right\}.
%\end{split}
\end{equation}
It is not difficult to see that the latter set is convex and closed.

\subsection{Characterization of $K(\Qb^\circ)$}
Let us first recall that the dual cone $\overline{\Cfrak}_+$ of the set of nonnegative matrix polynomials {is} the closure of the following set
\begin{equation}\label{dual_cone}
%\begin{split}
\Cfrak_+ := \left\{ \Sigmab \,:\, \innerprod{\Sigmab}{\Qb}>0 \text{ for all } \Qb\neq 0 \text{ such that } Q(e^{i\thetab})\geq 0 \ \forall \thetab\in\Tbb^d \right\}.
%\end{split}
\end{equation}
The normal cone at $\Qb^\circ$ in \eqref{thm_Rockafellar} is related to the dual cone as stated in the next lemma.

\begin{lemma}\label{lem_dual_cone}
	If $\Sigmab_K \in K(\Qb^\circ)$, then $-\Sigmab_K \in \overline{\Cfrak}_+$.
\end{lemma}
\begin{proof}
	Notice first that the set of $\Qb$ such that $Q\geq0$ is contained in our feasible set $\Lscr_+$. Take $\Sigmab_K \in K(\Qb^\circ)$ and rename $\Yb=-\Sigmab_K$. Then it means that $\innerprod{\Yb}{\Qb-\Qb^\circ} \geq 0 \textrm{ for all } \Qb\in\Lscr_+$, which implies that
	\begin{equation}\label{cond_Y}
	\innerprod{\Yb}{\Qb-\Qb^\circ} \geq 0 \quad \forall Q\geq0.
	\end{equation}
	Suppose that there exists some $Q\geq0$ such that $\innerprod{\Yb}{\Qb}<0$. Then the condition \eqref{cond_Y} can never hold since we can rescale the polynomial to make the inner product $\innerprod{\Yb}{\Qb}$ tend to $-\infty$. Therefore, we have $\innerprod{\Yb}{\Qb} \geq 0$ for all $Q\geq0$, and the assertion follows.
\end{proof}

Now take an arbitrary $\Sigmab_K \in K(\Qb^\circ)$. The above lemma implies that we have the representation
\begin{equation}\label{rep_normal_cone}
-\Sigmab_{K,\kb} = \int_{\Tbb^d} e^{i\innerprod{\kb}{\thetab}} \, \d M_K \quad \forall \kb\in\Lambda
\end{equation}
for some nonnegative matricial measure $M_K$ (cf.~\cite[Proposition 1, p.~1059]{Georgiou-06}). Our remaining task is to show the existence of such a measure under the same constraint \eqref{constraint_M_S} for $M_S$. This appears quite nontrivial, and we have only managed to achieve the result when the prior $\Psi$ {is {the inverse of} a matrix polynomial}, {i.e., a matricial spectral desity of the autoregressive type.}

\begin{proposition}
	If $\Psi=P^{-1}$ where $P(e^{i\thetab}):=\sum_{\kb\in\Lambda} P_\kb e^{-i\innerprod{\kb}{\thetab}}$ is a strictly positive matrix polynomial, then there exists a nonnegative matricial measure $M_K$ such that \eqref{rep_normal_cone} holds and $\trace \int_{\Tbb^d} (\nu \Psi^{-1} + Q^\circ) \d M_K = 0$.
\end{proposition}
\begin{proof}
	Again let us call $\Yb=-\Sigmab_K$ for convenience. By \cref{lem_dual_cone}, we know that there exists a nonnegative matricial measure such that $\Yb_{\kb} = \int_{\Tbb^d} e^{i\innerprod{\kb}{\thetab}} \, \d M_K$ for all $\kb\in\Lambda$. We can then rewrite the inner product $\innerprod{\Yb}{\Qb}=\trace\int_{\Tbb^d} Q\, \d M_K$. The condition for the normal cone means that $\trace\int_{\Tbb^d} (Q-Q^\circ)\, \d M_K\geq0$ for all $\Qb\in\Lscr_+$, which implies that $\forall \Qb\in\Lscr_+$,
	\begin{equation}\label{sandwich_M_K}
	\trace \int_{\Tbb^d} (\nu \Psi^{-1} + Q) \d M_K \geq \trace \int_{\Tbb^d} (\nu \Psi^{-1} + Q^\circ) \d M_K \geq 0.
	\end{equation}
	In particular, since now $\Psi=P^{-1}$, we can pick $Q=-\nu P + \varepsilon I$ for $\varepsilon>0$. Letting $\varepsilon\to0$, we get the desired trace-integral equality for the measure $M_K$. 
\end{proof}

In summary, when the prior has the form in the above proposition, we can apply \cite[Theorem~25.6]{rockafellar1970convex}, and see that
\begin{equation}
\zerob = \Sigmab - \Gamma(\Phi_\nu(\Qb^\circ)\d\mu + \d M_S) - \Gamma(\d M_K).
\end{equation}
Notice that here we have written the equality directly in terms of the vectors rather than {their} coordinates. In other words, the spectral measure $\Phi_\nu(\Qb^\circ)\d\mu + \d M_\s$ matches the give moments where the singular part $M_\s:=M_S+M_K$.
Moreover, we have the primal-dual complementarity
\begin{equation}\label{Primal-dual}
\Dcal_{1-\frac{1}{\nu}}(\d M,\Psi) = \Lcal_\nu(\d M,\Qb) \geq -J_\nu(\Qb).
\end{equation}
The equalities hold for $\d M=\Phi_\nu(\Qb^\circ)\d\mu + \d M_\s$ and $\Qb=\Qb^\circ$, and hence optimality follows.

\begin{remark}
	In the scalar case $(m=1)$, any object in the dual cone \eqref{dual_cone} admits an integral representation of the form $\sigma_\kb=\int_{\Tbb^d} e^{i\innerprod{\kb}{\thetab}} \d\lambda$ for $\kb\in\Lambda$ where the measure $\d\lambda$ is a sum of Dirac deltas \cite{lang1983spectral} and the number of impulses is determined by the cardinality of the index set $\Lambda$.
	{A similar characterization holds true for finite matricial covariance multisequences as reported in the recent paper \cite{Zhu-M2-LineSpec}.}
\end{remark}

\begin{remark}
	We briefly mention the case $\nu=1$ such that $\tau=1-\frac{1}{\nu}=0$. As reported in \cite{Z-15}, the divergence index $D_\tau$ in \eqref{tau_diverg} can be defined as $\tau\to 0$ via continuity, and one recovers the Itakura-Saito distance which in our formulation has the shape
	\begin{equation}\label{IS_dist}
	\Dcal_{0}(\d M,\Psi) = \int_{\Tbb^d} \trace \left\{ (\log\Psi-\log\Phi)\d\m + \Psi^{-1}\d M \right\} - m.
	\end{equation}
	A {discrete} version of the corresponding optimization problem has been studied in \cite{ZFKZ2019M2}. Following the lines in that paper and in \cref{sec:solution_dual} of the current work, one can show the existence and uniqueness of the solution to the dual problem when the spectrum is defined on the continuum $\Tbb^d$. However, in this specific case, we cannot devise an argument similar to those in \cref{prop_int_blowup} simply because we do not have the flexibility on the integer $\nu$. The main result is {summarized} in the next proposition.
\end{remark}

\begin{proposition}
	If \cref{assump_feasibility} and \cref{assump_Psi} hold, then the function
	\begin{equation}
	J_1(\Qb):=\langle\Qb,\Sigmab\rangle-\int_{\Tbb^d}\log\det(\Psi^{-1}+Q)\d\m
	\end{equation}
	is strictly convex over the feasible set $\Lscr_+$ and has a unique minimizer $\Qb^\circ$. If in addition the prior $\Psi=P^{-1}$, inverse of a strictly positive matrix polynomial with monomials indexed in $\Lambda$, then the problem to minimize 
	\eqref{IS_dist} over nonnegative matricial measures $\d M$ subject to the moment constraints has a solution of the form $(\Psi^{-1}+Q^\circ)^{-1}\d\m+\d M_\s\,$, 
	where $M_\s$ is singular {with respect to} the matricial measure $\m I$ and satisfies the equality $\trace \int_{\Tbb^d} (\Psi^{-1} + Q^\circ) \d M_\s = 0$.
\end{proposition}

\section{Conclusions}
\label{sec:conclusions}
We have considered the problem of estimating a M$^2$ (multidimensional and multivariate) spectral density of a random field using the covariance extension approach. The latter chooses as estimate the closest solution to the prior according to the $\tau$-divergence, and  matching the given set of covariances $\{\Sigma_\kb:\; \kb\in\Lambda\}$ computed from the data. The proposed paradigm is very flexible because through the parameter $\tau=1-\nu^{-1}$ with $\nu$ integer, we can impose some properties on the spectral estimator. More precisely, the corresponding optimization problem admits {a} solution which is not necessarily a spectral density, but rather a spectral measure that may contain a singular part. On the other hand, taking $\nu$ sufficiently large, such {a} solution is unique and is {guaranteed to be} a {rational} spectral density. {In future work we will consider the problem of connecting such spectra to system realization, e.g., using sum-of-squares representations (cf. \cite{dumitrescu2007positive}).}

\appendix

\section{Ancillary results}\label{sec:appendix_A}
The next theorem concerns the \emph{Lebesgue decomposition} of nonnegative matricial measures.
\begin{theorem}[A variant of Cram\'{e}r's theorem]\label{thm_Cramer_variant}
	Let $M$ be a nonnegative matricial measure, and let $\m$ be the Lebesgue measure on the family of Borel subsets of $\Tbb^d$. Then there exist unique matricial measures $M_\a$ and $M_\s$ such that $M = M_\a + M_\s$, $M_\a$ is strongly absolutely continuous with respect to $\m I$, $M_\s$ and $\m I$ are mutually singular, and $M_\a$ and $M_\s$ are nonnegative matricial measures.
\end{theorem}

The definition of strong absolute continuity for two matricial measures can be found in \cite[p.~361]{robertson1968decomposition}. In particular, under the context of \cref{thm_Cramer_variant}, it implies the existence of a Radon-Nikod\'{y}m derivative, i.e., there exists an integrable positive semidefinite matrix-valued function $\Phi$ such that
\begin{equation}
M_\ab(B) = \int_B \Phi \, \d\m
\end{equation}
for all Borel {sets $B\subset\Tbb^d$}. The function $\Phi$, which is often known as the \emph{density}, coincides with the element-wise Radon-Nikod\'{y}m derivative of $M_\ab$ {with respect to} the Lebesgue measure.

The concept of two matricial measures being mutually singular is defined in \cite[Section~6]{robertson1968decomposition}. Concerning \cref{thm_Cramer_variant}, it means that there {exists} a nonnegative measure $\lambda$ on $\Tbb^d$ such that both $M_\s$ and $\m I$ are absolutely continuous {with respect to} $\lambda$ (in an element-wise sense), and whenever $\frac{\d\m}{\d\lambda}(\thetab)\neq 0$, $M'_{\s,\lambda}(\thetab)=0$ up to a $\lambda$-null set.
Here $M'_{\s,\lambda}$ is the element-wise Radon-Nikod\'{y}m derivative of $M_\s$ {with respect to} $\lambda$.

\begin{remark}
	In the original theorem of Cram\'{e}r \cite{cramer1940theory}, $\m$ is the Lebesgue measure on the real line. The above variant holds true, as its proof is the same as that of \cite[Corollary~6.15]{robertson1968decomposition}, which follows directly from the more general result Theorem~6.14 in the same paper.
\end{remark}

\begin{lemma}\label{lem_tr_det}
	For a positive semidefinite matrix $A\in\Cbb^{m\times m}$, {the inequality}
	\begin{equation}\label{inequal_tr_det}
	\frac{1}{m}\trace(A^n) \geq (\det A)^{n/m}
	\end{equation}
	{holds} for any positive integer $n$.
\end{lemma}
\begin{proof}
	Let $A$ have eigenvalues $\lambda_k\geq0$, $k=1,\dots,m$. Given the arithmetic-geometric mean inequality
	\begin{equation}
	\left(\prod_{k=1}^m \lambda_k\right)^{1/m} \leq \frac{1}{m} \sum_{k=1}^m \lambda_k
	\end{equation}
	that holds for nonnegative real numbers, we have that
	\begin{equation}
	\frac{1}{m} \trace A \geq \left(\det A\right)^{1/m},
	\end{equation}
	and more generally,
	\begin{equation}
	\frac{1}{m} \trace(A^n) \geq \left[\det \left(A^n\right)\right]^{1/m} = \left(\det A\right)^{n/m}.
	\end{equation}
\end{proof}

\begin{proposition}\label{prop_Johan}
Let $p:\Rbb^n\to\Rbb$ be a polynomial with $p(\thetab_0)=0$ and assume that $p$ is nonnegative in a $\thetab_0$-centered ball $B_\varepsilon(\thetab_0)$ for some radius $\varepsilon>0$. Then if
$m\ge n/2$, we have that 
\[
\int_{B_\varepsilon(\thetab_0)}p(\thetab)^{-m}\d\m(\thetab)=\infty.
\]	
\end{proposition}

\begin{proof}
Note that since $p(\thetab)\ge 0$, the gradient vanishes at $\thetab_0$, i.e., $\nabla p(\thetab_0)=0$. Let the Hessian at $\thetab_0$ be $H=\Delta p(\thetab_0)$ and let $\alpha{>0}$ be such that $2 \alpha I>H$. By {the} continuity of the Hessian, there {exists} $\varepsilon'$ with $0<\varepsilon'<\varepsilon$ such that $2 \alpha I>\Delta p(\thetab)$ in $B_{\varepsilon'}(\thetab_0)$, and thus $p(\thetab)\le \alpha \|\thetab-\thetab_0\|^2$ in $B_{\varepsilon'}(\thetab_0)$.
Then we have 
\begin{align*}
\int_{B_\varepsilon(\thetab_0)}p(\thetab)^{-m}\d\m(\thetab)&\ge \alpha^{-m} \int_{B_{\varepsilon'}(\thetab_0)}\|\thetab-\thetab_0\|^{-2m}\d\m(\thetab)\\
& = {\alpha^{-m}} \int_{r=0}^{\varepsilon'}  r^{-2m}  S_{n-1} r^{n-1} \d r\\
& = {\alpha^{-m}} S_{n-1} \int_{r=0}^{\varepsilon'}  r^{n-1-2m} \d r,
\end{align*}
where $S_{n-1}=n\pi^{n/2}/\Gamma(1+n/2)$ is the area of a {unit} hypersphere in $\Rbb^n$.
Note that the integral on the right-hand side diverges if and only if the exponent $n-1-2m$ is less or equal to $-1$, i.e., when $m\ge n/2$.
\end{proof}

\begin{proposition}[Riesz representation]\label{Riesz_represent}
	A bounded real-valued linear functional $L$ on the space of Hermitian matrix-valued continuous functions on $\Tbb^d$ admits a representation
	\begin{equation}
	L(\Phi) = \trace \int_{\Tbb^d} \Phi \, \d M,
	\end{equation}
	where $M$ is a Hermitian matricial measure of bounded matricial variation.
\end{proposition}

\begin{proof}
	The assertion follows directly from the Riesz representation theorem for real-valued continuous functions. We first fix an orthonormal basis $\{X_k\}_{k=1}^{N}$ for the space of Hermitian matrices (of a fixed dimension), so that each Hermitian matrix-valued continuous function $\Phi$ can be written as $\Phi = \sum_{k=1}^{N} f_k X_k$,
	where the coordinates $f_k$ are real-valued continuous functions. Next, identify $L_k(f) := L(f X_k)$ as a linear functional on the space of continuous functions on $\Tbb^d$. Then by the scalar version of Riesz representation \cite[Theorem 6.19]{rudin1987real}, we have	
	\begin{equation}
	L_k(f) = L(f X_k) = \int_{\Tbb^d} f\, \d\lambda_k,
	\end{equation}
	where $\lambda_k$ is a regular Borel measure of bounded variation. Define $\d M:=\sum_{k=1}^{N} X_k \d\lambda_k$, and one can verify that
	\begin{equation}
	\begin{split}
	\trace \int_{\Tbb^d} \Phi \, \d M
	 & = \int_{\Tbb^d} \sum_{k=1}^{N} f_k \d\lambda_k \\
	 & = \sum_{k=1}^{N} L(f_k X_k) = L(\Phi),
	\end{split}
	\end{equation}
	where all the cross terms vanish because of the orthonormality of the basis matrices. The boundedness of such a measure $M$ follows from the boundedness of each scalar measure $\lambda_k$. It is not difficult to show that the norm of $L$ is upper bounded by $\trace |M|(\Tbb^d)$ where $|M|$ denotes the matricial total variation of $M$ (cf.~\cite{robertson1968decomposition}).
\end{proof}

\bibliographystyle{siamplain}
\bibliography{references}

\end{document}